\documentclass[12pt,leqno,a4paper]{article}
\usepackage[T1]{fontenc}
\usepackage{color}

\usepackage[pdfborder={0 0 0}]{hyperref}

\usepackage{amssymb}
\usepackage{amsmath}
\usepackage{amsthm}
\usepackage{mathrsfs}

\usepackage[shortlabels]{enumitem}

\newtheorem{theorem}{Theorem}
\numberwithin{theorem}{section}
\newtheorem{proposition}[theorem]{Proposition}
\newtheorem{lemma}[theorem]{Lemma}
\newtheorem{corollary}[theorem]{Corollary}

\theoremstyle{remark}

\theoremstyle{definition}
\newtheorem{remark}[theorem]{Remark}
\newtheorem{definition}[theorem]{Definition}
\newtheorem{example}[theorem]{Example}

\numberwithin{equation}{section}

\newcommand\set[1]{\left\{\,#1\,\right\}}		
\newcommand\abs[1]{\left|#1\right|}				
\newcommand\norm[1]{\left\Vert#1\right\Vert}	

\DeclareMathOperator{\id}{id}					
\DeclareMathOperator{\supp}{supp}				
\DeclareMathOperator{\tr}{tr}					
\DeclareMathOperator{\curl}{curl}				
\DeclareMathOperator{\divv}{div}				

\def\R{\mathbb{R}}

\def\T{\mathbb{T}}

\newcommand{\cC}{{\mathcal C}}

\newcommand{\cH}{{\mathcal H}}

\newcommand{\cS}{{\mathcal S}}

\newcommand{\lamax}{\lambda_{\text{max}}}
\newcommand{\lamin}{\lambda_{\text{min}}}						

\usepackage[numbers]{natbib}

\begin{document}
\title{On bounded two-dimensional globally dissipative Euler flows}
\author{Bj\"orn Gebhard \quad J\'ozsef J. Kolumb\'an}
\date{}
\maketitle

\begin{abstract}
We examine the two-dimensional Euler equations including the local energy (in)equality as a differential inclusion and show that the associated relaxation essentially reduces to the known relaxation for the Euler equations considered without local energy (im)balance. Concerning bounded solutions we provide a sufficient criterion for a globally dissipative subsolution to induce infinitely many globally dissipative solutions having the same initial data, pressure and dissipation measure as the subsolution. The criterion can easily be verified in the case of a flat vortex sheet giving rise to the Kelvin-Helmholtz instability. 
As another application we show that there exists initial data, for which associated globally dissipative solutions realize every dissipation measure from an open set in $\cC^0(\T^2\times[0,T])$. In fact the set of such initial data is dense in the space of solenoidal $L^2(\T^2;\R^2)$ vector fields.
\end{abstract}
%
%

\section{Introduction}

We consider the homogeneous incompressible Euler equations
\begin{align}\label{eq:euler}
\begin{split}
\partial_t v +\divv v\otimes v +\nabla p&=0,\\
\divv v&=0
\end{split}
\end{align}
on the two-dimensional torus $\T^2$. Here $v:\T^2\times(0,T)\rightarrow\R^2$ is the velocity field and $p:\T^2\times(0,T)\rightarrow\R$ the pressure of the fluid, $T>0$. For $v$ we fix  initial data $v_0:\T^2\rightarrow\R^2$ satisfying $\divv v_0=0$.

Multiplying the momentum balance \eqref{eq:euler} by $v$ it follows that if $(v,p)$ is a sufficiently smooth solution to \eqref{eq:euler}, then the associated local energy density $e:\T^2\times(0,T)\rightarrow\R$, $e(x,t)=\frac{1}{2}\abs{v(x,t)}^2$ satisfies the equation
\begin{equation}\label{eq:energy_pde}
\partial_t e+\divv ((e+p)v)=0,
\end{equation} 
from which one concludes the conservation of the total energy
\begin{equation}\label{eq:total_energy}
E(t):=\int_{\T^2} e(x,t)\:dx.
\end{equation}
In three space dimensions the precise threshold below which the conservation of total energy is violated has already been conjectured by Onsager \cite{Onsager} to be $\cC^{1/3}$. As is famously known this is no longer a conjecture. The conservative part has been shown by Constantin, E, Titi \cite{Constantin_E_Titi}, while counterexamples of class $\cC^{\alpha}$, $\alpha<\frac{1}{3}$ have been constructed by Isett \cite{Isett_Onsager}, and improved to examples with strict dissipation by Buckmaster, De Lellis, Sz\'ekelyhidi, Vicol \cite{Buckmaster_DeL_Sz_Vicol}. These works are relying upon a series of works taking its beginning with the articles of De Lellis and Sz\'ekelyhidi \cite{DeL-Sz-Annals,DeL-Sz-Adm,DeL-Sz-Continuous} who introduced the method of convex integration to the context of fluid dynamics. We refer to the recent surveys \cite{Buckmaster_Vicol_survey,DeL-Sz-Survey} for more details.

Solutions obtained by means of convex integration have not only served as counterexamples, but, due to their highly oscillatory nature, have also been utilized to describe naturally occurring turbulent behaviour in fluids. Examples include turbulence emanating from vortex sheet initial data for the homogeneous Euler equations \cite{Mengual_Sz_vortex_sheet,Sz-KH}, and from the mixture of two different density fluids due to gravity, see \cite{Castro_Cordoba_Faraco,Castro_Faraco_Mengual,Castro_Faraco_Mengual_2,Cordoba,Foerster-Sz,Hitruhin-Lindberg,Mengual,Noisette-Sz,Sz-Muskat} for the incompressible porous media equation as underlying model and \cite{GK,GKSz} for the inhomogeneous Euler equations. The construction of these solutions crucially relies on an explicit relaxation of the differential inclusion associated with the considered partial differential equations.

In the case of the homogeneous or inhomogeneous Euler equations the solutions obtained in \cite{GK,GKSz,Mengual_Sz_vortex_sheet,Sz-KH} enjoy the property of weak admissibility. Meaning that in view of observed anomalous dissipation, as well as weak-strong uniqueness, the conservation of total energy  is relaxed to the condition $E(t)\leq E(0)$, $t\geq 0$, see Section \ref{sec:globally_dissip} below for precise definitions. However, even the stronger condition $E(s)\leq E(t)$, $s\geq t\geq 0$ does not exclude the unphysical case of spatially localized creation of energy getting compensated by dissipation on a different set, such that after integration the total energy is still decaying. In fact, as for example  the present investigation shows, there exist plenty of solutions having this property. 
In order to exclude local creation of energy we therefore add in the present paper the local energy balance \eqref{eq:energy_pde}, or rather its dissipative analogue with $\leq$ instead of equality, to the differential inclusion of the Euler equations and examine its relaxation. 

We point out that it is known that imposing the local energy (in)equality does not recover any well-posedness for the Cauchy-problem. Essentially bounded counterexamples for any space dimension $n\geq 2$ have first been provided by De Lellis and Sz\'ekelyhidi \cite{DeL-Sz-Adm},  see Section \ref{sec:comparison_to_without_energy} for more details. In dimension $3$ the regularity has later been improved to H\"older continuous solutions with any exponent $\alpha <\frac{1}{15}$ by Isett \cite{Isett_globally_dissipative} and with any exponent $\alpha <\frac{1}{7}$ by De Lellis and Kwon \cite{DeL-Kwon}. 

In contrast, the present article does not contribute to the question of best possible regularity. However, it addresses the two-dimensional $L^\infty$ realm and stems its motivation from the use of solutions obtained by convex integration as description for turbulent flows, where as indicated earlier the relaxation of the associated differential inclusion plays an important role. In particular, the present work sheds more light on the importance of convexity in establishing the relaxation, c.f. the calculation of the ($\Lambda$-)convex hull of the associated set of nonlinear constraints in Sections \ref{sec:convex_hull_unconstraint} and \ref{sec:infty}.

In the following subsections we present our results and compare them to the known case when the energy (in)equality is not viewed as part of the differential inclusion. Section \ref{sec:convex_integration} carries out the convex integration in the Tartar framework, where in Section \ref{sec:convex_hull_unconstraint} we first compute the associated $\Lambda$-convex hull in the unconstraint case, and thereafter in Section \ref{sec:infty} introduce a $L^\infty$ bound. Section \ref{sec:density_proof} addresses existence and density of initial data for which induced solutions realize every local dissipation measure from an open set of continuous functions.

\subsection{Globally dissipative solutions}\label{sec:globally_dissip}

We define weak solutions to \eqref{eq:euler} and two different notions of admissibility as follows, see also \cite[Section 2.1]{DeL-Sz-Adm} for an overview of different notions of admissibility.
\begin{definition}\label{def:weaksolseu}
Let $v_0\in L^2(\T^2;\R^2)$ be a weakly divergence-free vector field.
We say that $v\in L^2(\T^2\times (0,T);\R^2)$ is a weak solution to equation \eqref{eq:euler} with initial data $v_0$ if for any test functions $\Phi\in \cC^\infty_c(\T^2\times[0,T);\mathbb{R}^2) $ with $\divv \Phi=0$, $\Psi\in \cC^\infty_c(\T^2\times[0,T))$ we have
\begin{align}\label{eq:weak_form_of_momentum}
\int_0^T\int_{\T^2} \left[ v \cdot \partial_t \Phi + v\otimes v :\nabla\Phi \right] \:dx\:dt +\int_{\T^2}v_0 (x)\cdot \Phi(x,0)\ dx=0,\\\label{eq:weak_divegence_free}
\int_0^T\int_{\T^2} v\cdot\nabla \Psi\:dx\:dt=0.
\end{align}
Moreover, $v$ is called \emph{weakly admissible} provided $v\in \cC^0([0,T);L^2_w(\T^2;\R^2))$ and
\begin{equation}\label{eq:weak_energy_inequality}
E(t):=\int_{\T^2}\abs{v(x,t)}^2\:dx\leq \int_{\T^2}\abs{v_0(x)}^2\:dx\quad \text{for all }t\in[0,T);
\end{equation}
and \emph{globally dissipative}, if $v\in L^3(\T^2\times(0,T))$ and for every nonnegative testfunction $\Psi\in\cC_c^\infty(\T^2\times[0,T),\R_+)$ there holds
\begin{align}\label{eq:local_energy_inequality_weak_version}
\int_0^T\int_{\T^2} \left[|v|^2 \partial_t \Psi + \left(|v|^2+2p\right)v\cdot\nabla\Psi \right] \ dx \ dt +\int_{\T^2} |v_0(x)|^2 \Psi(x,0)\ dx\geq 0.
\end{align}
The apriori only distributional defect in equation \eqref{eq:local_energy_inequality_weak_version}, i.e. 
\[
\mu:=\partial_t\left(\frac{\abs{v}^2}{2}\right)+\divv\left( \left(\frac{\abs{v}^2}{2}+p\right)v\right)\in \big(\cC^\infty_c(\T^2\times[0,T))\big)^*
\]
is called the \emph{energy dissipation measure}.
\end{definition}
\begin{remark}\label{rem:properties_of_weak_solutions}
a) It is well-known that the pressure, which is unique up to a function depending on time only, can be recovered from the notion of weak solution given above, cf. \cite{Temam}. If $v$ is only a weak solution, then $p$ is apriori only a distribution, whereas the additional integrability $v\in L^3(\T^2\times(0,T))$ in the globally dissipative case allows to conclude that $p\in L^{\frac{3}{2}}(\T^2\times (0,T))$ and thus the integral in \eqref{eq:local_energy_inequality_weak_version} is well-defined. Note also that in view of \eqref{eq:weak_divegence_free} a change of $p(x,t)$ to $p(x,t)+\tilde{p}(t)$ does not affect \eqref{eq:local_energy_inequality_weak_version}.

\noindent b) In the present work we will mostly deal with essentially bounded globally dissipative solutions, i.e. $v\in L^\infty(\T^2\times (0,T))$. In that case \cite[Lemma 8]{DeL-Sz-Adm} applied to the the momentum balance \eqref{eq:weak_form_of_momentum} shows that w.l.o.g. $v\in \cC^0([0,T);L^2_w(\T^2;\R^2))$. Moreover, if in addition the dissipation measure $\mu$ is not only a distribution, but a $L^1(\T^2\times(0,T))$ function, then \cite[Lemma 8]{DeL-Sz-Adm} also shows that for any $\varphi\in L^2(\T^2)$ the map $[0,T)\ni t\mapsto \int_{\T^2}\varphi(x)\abs{v(x,t)}^2\:dx\in\R$ can assumed to be continuous. In fact this is not stated in the cited lemma, but the only change in the proof is a new term in \cite[equation (89)]{DeL-Sz-Adm}. Taking in particular $\varphi\equiv 1$ we see that the total energy $E(t)$ is continuous on $[0,T)$. The local energy inequality \eqref{eq:local_energy_inequality_weak_version} then shows $E(t)\leq E(s)$ for all $0\leq s\leq t<T$. Thus $v$ is in particular weakly admissible.

\noindent c) Weakly admissible solutions enjoy the weak-strong uniqueness property \cite{Brenier_DeL_Sz,Wiedemann}.
\end{remark}

\subsection{The associated differential inclusion}\label{sec:differential_inclusion}
Our goal is to provide the relaxation of the above notion of globally dissipative weak solutions when viewed as a differential inclusion. For this purpose we rewrite the Euler equations including the local energy inequality as the linear system 
\begin{align}\label{eq:linear_system_with_p}
\begin{split}
\partial_tv+\divv \sigma +\nabla (e+p)&=0,\\
\divv v&=0,\\
\partial_t e+\divv m &=\mu,
\end{split}
\end{align}
coupled with the following set of nonlinear pointwise constraints:
\begin{align}\label{eq:pointwise_constraints_when_introducing_diff_inclusion}
    \sigma=(v\otimes v)^\circ ,\quad e=\frac{1}{2}\abs{v}^2,\quad m=(e+p)v.
\end{align}
Here the tuple $(v,m,\sigma,e,p)$ takes values in $\R^2\times\R^2\times\cS_0^{2\times 2}\times\R\times\R$, where $\cS_0^{2\times 2}$ denotes the space of traceless symmetric $2$-by-$2$ matrices. Note that in dimension $2$ we can split 
\[
v\otimes v=(v\otimes v)^\circ +\frac{1}{2}\abs{v}^2\id=\sigma +e\id, 
\]
whereas in dimension $d\geq 3$ one would get a factor $\frac{2}{d}$ in front of the $e$ in the first equation of \eqref{eq:linear_system_with_p}. This case is not considered here.

Furthermore, the energy dissipation measure $\mu$ is assumed to be a distribution that can be tested against functions from $\cC_c^\infty(\T^2\times [0,T))$. We consider it to be given and to be negative in the sense that $\mu[\Psi]\leq 0$ for all positive test functions $\Psi\in\cC^\infty_c(\T^2\times[0,T);\R_+)$. The condition $\mu\leq 0$ is not really needed in the investigation, for instance Theorem \ref{thm:main_theorem} below remains true with a general $\mu$, but due to our interest in the physically relevant regime, we exclude any local creation of energy.

System \eqref{eq:linear_system_with_p} typically is accompanied with initial data $(v_0,e_0)$, where $v_0$ is divergence-free.
The following definition contains the precise notion of solution to \eqref{eq:linear_system_with_p}. For simplicity it is formulated in the essentially bounded case. For more generality one would need to adapt the integrability for the different quantities, e.g. $v\in L^3$, $m\in L^1$, $e,\sigma,p\in L^{\frac{3}{2}}$.

\begin{definition}\label{def:weaksolslin}
Let $v_0\in L^\infty(\T^2;\R^2)$ be a weakly divergence-free vector field, $e_0:=\frac{1}{2}\abs{v_0}^2$ and $\mu\in\left(\cC_c^\infty(\T^2\times [0,T))\right)^*$ be negative. 
We say that the tuple $(v,m,\sigma,e) \in L^\infty(\T^2\times(0,T);\R^2\times\R^2\times\cS_0^{2\times 2}\times\R)$ is a weak solution to \eqref{eq:linear_system_with_p} with initial data $(v_0,e_0)$ if for any test functions $\Phi\in \cC^\infty_c(\T^2\times[0,T);\mathbb{R}^2) $ with $\divv \Phi =0$,  $\Psi\in \cC^\infty_c(\T^2\times[0,T)) $ we have
\begin{align*}
\int_0^T\int_{\T^2} \left[ v \cdot \partial_t \Phi + \sigma :\nabla\Phi \right] \ dx \ dt +\int_{\T^2}v_0 (x)\cdot \Phi(x,0)\ dx&=0,\\
\int_0^T\int_{\T^2} v\cdot\nabla \Psi\:dx\:dt&=0,\\
\int_0^T\int_{\T^2} \left[e \partial_t \Psi + m\cdot\nabla\Psi\right] \ dx \ dt +\int_{\T^2} e_0(x) \Psi(x,0)\ dx&=-\mu[\Psi].
\end{align*}
\end{definition}
As in Remark \ref{rem:properties_of_weak_solutions} for weak Euler solutions one can also here recover the pressure $p$ up to a function depending only on time by solving $-\Delta p=\divv \divv \left(\sigma+e\id\right)$. In the here considered case Calderon-Zygmund implies $p\in L^q(\T^2\times (0,T))$ for all $q\in (1,\infty)$. Hence if $(v,m,\sigma,e)$ with induced pressure $p$ is a solution of \eqref{eq:linear_system_with_p} in the sense of Definition \ref{def:weaksolslin} and if \eqref{eq:pointwise_constraints_when_introducing_diff_inclusion} holds pointwise for a.e. $(x,t)\in\T^2\times(0,T)$, then $v$ is a globally dissipative solution of the Euler equations in the sense of Definition \ref{def:weaksolseu}. Observe that also here the non-uniqueness of the pressure does not play a role, since $m(x,t)$ can be changed to $m(x,t)+\tilde{p}(t)v(x,t)$ if needed.

\subsection{A convex integration theorem}

In Section \ref{sec:convex_integration} we will investigate the differential inclusion \eqref{eq:linear_system_with_p}, \eqref{eq:pointwise_constraints_when_introducing_diff_inclusion} and in particular compute its relaxation, see Proposition \ref{prop:hull}. Regarding  essentially bounded solutions we have the following statement. Here $\lamax(M)$ denotes the maximal eigenvalue of $M\in \R^{2\times 2}$ symmetric.

\begin{theorem}\label{thm:main_theorem} Suppose that there exists an $L^\infty$ solution $(v,m,\sigma,e)$ of \eqref{eq:linear_system_with_p} with initial data $(v_0,e_0)$, $2e_0=\abs{v_0}^2$, induced pressure $p$ and negative dissipation measure $\mu\in \left(\cC_c^\infty(\T^2\times [0,T))\right)^*$. Suppose further that to this solution there exists an open set $\mathscr{U}\subset\T^2\times(0,T)$, as well as $\varepsilon>0$, such that $(v,m,\sigma,e)$ is continuous on $\mathscr{U}$, $p\in L
^\infty(\mathscr{U})\cap \cC^0(\mathscr{U})$ and such that there holds
\begin{align}\label{eq:sufficient_condition_for_Linfty_hull}
    \lamax(v\otimes v-\sigma)+\varepsilon\abs{m-(e+p)v}<e\text{ on }\mathscr{U},
\end{align}
as well as \eqref{eq:pointwise_constraints_when_introducing_diff_inclusion} almost everywhere in $\T^2\times(0,T)\setminus\mathscr{U}$.
Then there exist infinitely many globally dissipative weak solutions $v_{sol}\in L^\infty(\T^2\times (0,T))$ of \eqref{eq:euler} with initial data $v_0$, pressure $p$ and dissipation measure $\mu$, i.e. there holds 
\begin{align*}
    \partial_tv_{sol}+\divv\left(v_{sol}\otimes v_{sol}\right)+\nabla p&=0,\\
    \divv v_{sol} &=0,\\
    \partial_t\left(|v_{sol}|^2\right) + \divv\left(\left(|v_{sol}|^2+2p\right)v_{sol}\right)&=2\mu.
\end{align*}
On $\T^2\times(0,T)\setminus \mathscr{U}$ the solutions $v_{sol}$ coincide with $v$. Furthermore, among the infinitely many solutions one can find a sequence $(v_k,m_k,\sigma_k,e_k)$, where $e_k:=\frac{1}{2}\abs{v_k}^2$, $m_k:=(e_k+p)v_k$, $\sigma_k:=(v_k\otimes v_k)^\circ$, such that $v_k\rightharpoonup v$, $m_k\rightharpoonup m$, $\sigma_k\rightharpoonup \sigma$, $e_k\rightharpoonup e$ weakly in $L^2(\T^2\times(0,T))$ as $k\rightarrow \infty$.
\end{theorem}

\begin{remark}\label{rem:only_sufficient_condition} a) As usual we refer to $\mathscr{U}$ as the turbulent zone of the solutions.\\
\noindent b) It would be enough to assume the boundedness of $(v,m,\sigma,e)$ on $\mathscr{U}$ and for instance $v\in L^3(\T^2\times(0,T)\setminus \mathscr{U};\R^2)$. Then there still exist globally dissipative solutions as in Theorem \ref{thm:main_theorem} except that they would only be bounded on $\mathscr{U}$.  
\\
\noindent c) We like to emphasize that condition \eqref{eq:sufficient_condition_for_Linfty_hull} is only a sufficient condition for inducing infinitely many bounded globally dissipative solutions. In other words the tuple $(v,m,\sigma,e,p)$ is a subsolution of the system \eqref{eq:linear_system_with_p}, \eqref{eq:pointwise_constraints_when_introducing_diff_inclusion}, but not every subsolution satisfies \eqref{eq:sufficient_condition_for_Linfty_hull}. For more details we refer the reader to Section \ref{sec:convex_integration}, in particular to Propositions \ref{prop:hull}, \ref{prop:condition_for_being_in_K_gamma_hull} and Theorem \ref{thm:abstract_CI_theorem}.\\
\noindent d) For the $v$- and $e$-component the weak $L^2(\T^2\times (0,T))$ convergence of solutions to the subsolution can be improved to $\cC^0([0,T],L^2_w(\T^2))$ convergence, i.e. weak convergence on every time slice, by using in Section \ref{sec:conclusion} below the shifted grid method from \cite{DeL-Sz-Adm}.
\end{remark}

\subsection{Comparison to the relaxation without energy inequality}\label{sec:comparison_to_without_energy}

We recall that the differential inclusion of the Euler system without energy inequality considered in \cite{DeL-Sz-Annals} consisted only of the linear system 
\begin{align}
\begin{split}\label{eq:normal_euler_linear_system}
    \partial_tv+\divv\sigma+\nabla(e+ p)&=0,\\
    \divv v&=0,
\end{split}
\end{align}
(note that $e+p$ can be redefined to a new pressure $q$) 
coupled with the set of nonlinear pointwise constraints
\begin{equation}\label{eq:normal_euler_pointwise}
v\otimes v-\sigma =e\id,
\end{equation}
where in contrast to the present paper the function $e(x,t)$ is a fixed function, prescribing the kinetic energy density of the associated solutions. In order to eliminate unphysical behaviour the given function $e$ is supposed to satisfy 
\begin{align*}
\int_{\T^2} e(x,t) \, dx \leq \int_{\T^2} e(x,0) \, dx\text{ for any }t\geq 0,
\end{align*}
such that corresponding solutions are weakly admissible in the sense of Definition \ref{def:weaksolseu}.

In \cite{DeL-Sz-Adm} it was shown that the relaxation of the pointwise constraints \eqref{eq:normal_euler_pointwise} with respect to \eqref{eq:normal_euler_linear_system} is given by the convex condition
\begin{equation}\label{eq:normal_euler_relaxation}
\lambda_{\max}(v\otimes v-\sigma)<e.
\end{equation}
Tuples $(v,\sigma,p)$ satisfying the linear system \eqref{eq:normal_euler_linear_system} and the pointwise inequality \eqref{eq:normal_euler_relaxation} are called Euler subsolutions with respect to the given energy profile $e(x,t)$. They induce infinitely many Euler solutions with energy density $e(x,t)$. 

The relaxation with respect to a fixed $e(x,t)$ has been used in \cite{DeL-Sz-Adm} to show the earlier mentioned non-uniqueness of globally dissipative solutions in the class of essentially bounded solutions. That the local energy inequality is satisfied for those solutions relies on the special choice of $e$, which is $e(x,t)=\chi_{\Omega'}(x)\hat{e}(t)$, $\chi_{\Omega'}$ denoting the indicator function of an open set $\Omega'\subset\T^2$ and $\hat{e}:[0,T]\rightarrow\R_+$ suitably chosen, see \cite[Section 6.1]{DeL-Sz-Adm}. For a more general $e$, as for example occuring  for vortex sheet initial data, see Section \ref{sec:sheet}, the induced solutions are not automatically globally dissipative. 

In order to overcome this we consider the extended inclusion \eqref{eq:linear_system_with_p}, \eqref{eq:pointwise_constraints_when_introducing_diff_inclusion}. The differences to \eqref{eq:normal_euler_linear_system}, \eqref{eq:normal_euler_pointwise} are that first of all $e$ is no longer a given function and instead also considered as part of the variables. Second the inequality 
\begin{equation}\label{eq:normal_euler_energy_inequ}
\partial_t e+\divv m\leq 0
\end{equation}
is added to \eqref{eq:normal_euler_linear_system}, as well as the constraint
\[
m=(e+p)v
\]
to \eqref{eq:normal_euler_pointwise}. This way we arrive at \eqref{eq:linear_system_with_p}, \eqref{eq:pointwise_constraints_when_introducing_diff_inclusion} except that there the local dissipation is specified by $\mu$. 

As Proposition \ref{prop:hull} will show, it turns out that although the new constraint $m=(e+p)v$ has been added, it is not seen in the pointwise relaxation of the differential inclusion. That is the interior of the $\Lambda$-convex hull associated with \eqref{eq:linear_system_with_p}, \eqref{eq:pointwise_constraints_when_introducing_diff_inclusion} consists of the tuples $(v,m,\sigma,e,p)$ satisfying \eqref{eq:normal_euler_relaxation}, cf. Remark \ref{rem:full_hull}.
Note that this set is clearly unbounded. 
Condition \eqref{eq:sufficient_condition_for_Linfty_hull} in Theorem \ref{thm:main_theorem} has been added in order to apply convex integration in the usual $L
^\infty$-setting. As mentioned earlier more details can be found in Section \ref{sec:convex_integration}.

\subsection{Vanishing viscosity limits}
Globally dissipative solutions occur as vanishing viscosity limits of suitable weak solutions in the sense of Caffarelli, Kohn, Nirenberg \cite{Caffarelli_Kohn_Nirenberg} of the Navier-Stokes equations 
\[
\partial_tv_\nu+\divv(v_\nu\otimes v_\nu)+\nabla p_\nu=\nu \Delta v_\nu,\quad \divv v_\nu=0,
\]
provided the convergence happens to be strong in $L^3(\T^2\times(0,T))$. In the two-dimensional case considered here, the unique Leray-Hopf solutions automatically satisfy the local energy balance
\[
\partial_t e_\nu+\divv\left((e_\nu+p_\nu)v_\nu\right)-\nu \Delta e_\nu=-\nu \abs{\nabla v_\nu}^2,\quad e:=\frac{1}{2}\abs{v_\nu}^2,
\]
see \cite[Proposition 5]{Duchon_Robert} by Duchon, Robert. Therefore they are in particular suitable and any strong $L^3$ limit as $\nu\rightarrow 0$ is a globally dissipative Euler solution. 

If we instead of strong convergence assume that $v_\nu$, $p_\nu$, as well as the components
\[
e_\nu:=\frac{1}{2}\abs{v_\nu}^2,\quad \sigma_\nu:=(v_\nu\otimes v_\nu)^\circ,\quad m_\nu:=(e_\nu+p_\nu)v_\nu,
\]
were to converge weakly to a tuple $(v,m,\sigma,e,p)$ belonging to $L^1$, then the limit satisfies \eqref{eq:normal_euler_linear_system}, \eqref{eq:normal_euler_energy_inequ} and 
\[
\lamax(v\otimes v-\sigma)\leq e\quad a.e. \text{ on }\T^2\times(0,T)
\]
due to convexity. In view of Remark \ref{rem:full_hull} the latter is the only pointwise condition one can deduce from convexity without any further information.

\subsection{The flat vortex sheet}\label{sec:sheet}
In \cite{Sz-KH} Sz\'ekelyhidi has constructed, or rather selected as viscosity solutions of Burgers equation, a family of Euler subsolutions emanating from a flat vortex sheet with initial velocity 
\[
v_0(x)=\begin{cases}
-e_1,&x_2<0,\\
e_1,&x_2>0,
\end{cases}
\]
where $x\in \T^2=\left[-\frac{1}{2},\frac{1}{2}\right]^2/_{\sim}$ and $e_1=(1,0)$. Out of this family the subsolution maximizing the total energy dissipation, see \cite[Remarks]{Sz-KH}, consists of the tuple
\begin{gather*}
    v=\begin{pmatrix}
    \alpha\\0
    \end{pmatrix},\quad \sigma=\begin{pmatrix}
    \frac{\alpha^2}{2} &-\frac{1-\alpha^2}{4}\\
    -\frac{1-\alpha^2}{4} &-\frac{\alpha^2}{2}
    \end{pmatrix},\quad p=\frac{\alpha^2}{2}-e,
\end{gather*}
where the prescribed energy density is given by 
\[
e=\frac{1}{2}- \frac{\delta(1-\alpha^2)}{4},~\delta\in[0,1)
\]
and the function 
$\alpha:\T^2\times[0,\infty)\rightarrow\R$ reads
\[
\alpha(x,t)=\begin{cases}
-1,&x_2<-\frac{t}{2},\\
\frac{x_2}{2t},&-\frac{t}{2}<x_2<\frac{t}{2},\\
1,&x_2>\frac{t}{2}.
\end{cases}
\]

Let us first discuss the consequences of convex integrating without the local energy inequality. Due to \cite{DeL-Sz-Adm} the just stated subsolution induces infinitely many solutions $(v_{sol},p_{sol})$ with local energy density $\frac{1}{2}\abs{v_{sol}}^2=e$ a.e. and pressure $p_{sol}=p$, while the local form of the energy balance 
\[
\partial_t e+\divv ((e+p)v_{sol})=:\mu_{sol}
\]
is not explicitly known for the solutions and could apriori be positive for some $(x,t)$ in the turbulent zone $\set{(x,t):2\abs{x_2}<t}$. Note however that, since the induced solutions can be found arbitrarily close to the subsolution with respect to the weak $L^2$ topology, for selected sequences of solutions there holds
\[
\mu_{sol}[\Psi]\rightarrow \big(\partial_t e+\divv((e+p)v)\big)[\Psi]=\partial_te[\Psi]
\]
as $v_{sol}\rightharpoonup v$ in $L^2(\T^2\times(0,T))$, when tested against a fixed $\Psi\in\cC^\infty_c(\T^2\times[0,T))$. Exploiting this directly in their convex integration strategy Mengual and Sz\'ekelyhidi have constructed in \cite{Mengual_Sz_vortex_sheet} weakly admissible solutions which do not have local energy creation on spatial length scales larger than a beforehand choosable constant, see \cite[Theorem 1.3]{Mengual_Sz_vortex_sheet}. The construction in \cite{Mengual_Sz_vortex_sheet} addresses in fact the more general case of an arbitrary sufficiently regular initial vortex sheet. The used convex integration strategy relies on \cite{Castro_Faraco_Mengual,DeL-Sz-Adm}.
 
Now we turn to convex integration with local energy inequality. Let $v,\sigma,p,e$ be as above and set 
\[
m:=(e+p)v,\quad \mu:=\partial_t e. 
\]
Then using for instance $\varepsilon=1$ the following statement can easily be verified.
\begin{example}\label{ex:flat_KH}
The tuple $(v,m,\sigma,e)$ together with the pressure $p$ and dissipation measure $\mu$ is a globally dissipative subsolution in the sense that it satisfies the conditions of Theorem \ref{thm:main_theorem}.
\end{example}
In consequence we find globally dissipative solutions $(\tilde{v}_{sol},\tilde{p}_{sol}=p)$ with local dissipation measure $\tilde{\mu}_{sol}=\mu$. 
Note on the other hand that  $\abs{\tilde{v}_{sol}}^2=2e$ a.e. does no longer need to hold true, but similar to the previous case for the dissipation measure one still has $\tilde{e}_{sol}\rightharpoonup e$ in $L^2$ (or in $\cC^0([0,T],L^2_w(\T^2))$ cf. Remark \ref{rem:only_sufficient_condition} d)) for suitably chosen sequences. Therefore one could say that in the upgrade from weakly admissible to globally dissipative solutions the local dissipation measure $\mu$ has taken the role of the local energy density.

Finally we remark that the total energy dissipation coincides for both types for solutions, i.e.
\[
\frac{d}{dt}\int_{\T^2}\abs{\tilde{v}_{sol}}^2\:dx=2\int_{\T^2}\mu\:dx=2\int_{\T^2}\partial_t e\:dx=\frac{d}{dt}\int_{\T^2}\abs{v_{sol}}^2\:dx.
\]



\subsection{A density result}
As another application of the relaxation we conclude a result of ``density of wild initial data''-type. Wild initial data is initial data for which the Cauchy problem in the considered regularity and/or admissibility class admits infinitely many solutions. With the onset of convex integration in fluid dynamics it was possible to show that wild initial data is dense in the set of all divergence-free velocity fields, see \cite{Daneri-Runa-Sz,Daneri-Sz,Sz-MPI,Sz-Wiedemann} for weakly admissible solutions of different regularity classes with \cite{Daneri-Runa-Sz} by Daneri, Runa, Sz\'ekelyhidi covering the $\cC^{1/3-}$ case in three space dimensions. The question has also been addressed for weakly admissible solutions of the compressible Euler equations, see \cite{Chen-Vasseur-Yu,Feireisl-Klingenberg-Markfelder}.

In the context of globally dissipative solutions, either incompressible or compressible, so far the only result respecting the local energy (in)equality is the one by Isett \cite{Isett_globally_dissipative}. He shows that the set of divergence-free vector fields inducing infinitely many $\cC^\alpha(I\times \T^3)$, $\alpha<1/15$ solutions with zero dissipation $\mu=0$ on some time interval $I\ni 0$ depending on the initial data is dense in the $\cC^0(\T^3)$ topology.

We complement this result in the 2D case by showing the $L^2$-density of wild initial data for bounded globally dissipative solutions. In addition to the usual wildness, each initial data from the dense set constructed here leads to infinitely many solutions having any arbitrary dissipation measure from a sufficiently small open set in $\cC^0(\T^2\times[0,T])$. Here $T>0$ is uniform. 

More precisely, we fix $T>0$ and for $\delta>0$ we define a set of vector fields $A_\delta$ by saying that $v_0\in A_\delta$ if and only if $v_0\in L^2(\T^2)$, $\divv v_0=0$ and for every $\mu\in \cC^0(\T^2\times[0,T])$ with $-\delta/T<\mu\leq 0$ there exist infinitely many essentially bounded globally dissipative solutions with initial data $v_0$ and dissipation measure $\mu$. Then there holds
\begin{theorem}\label{thm:dense}
The set $A_\delta$ is $\sqrt{20\delta}$-dense in the space of divergence-free $L^2(\T^2;\R^2)$ vector fields. 
\end{theorem}
As the proof of Theorem \ref{thm:dense} in Section \ref{sec:density_proof} shows, all of the induced solutions have the property that the turbulent zone instantaneously consists of all of $\T^2$, i.e. $\mathscr{U}=\T^2\times(0,T]$.

\section{Convex integration via the Tartar framework}\label{sec:convex_integration}

The Tartar framework, originally introduced in the context of compensated compactness \cite{Tartar}, in combination with a Baire category argument is by now a well-known procedure for convex integration. To prove Theorem \ref{thm:main_theorem} we will use a version for differential inclusions when the set of nonlinear constraints is not constant (c.f. e.g. \cite{Crippa}). 

The differential inclusion has already been introduced in Section \ref{sec:differential_inclusion}. But to be more precise, we only consider $(v,m,\sigma,e)$ as variables while $p$ and $\mu$ are considered as given. That is, we fix a pressure function and a negative dissipation measure
\begin{align}\label{eq:fixed_pressure_and_mu}
 p:\T^2\times (0,T)\rightarrow \R,\quad \mu\in\left(\cC^\infty_0(\T^2\times [0,T)\right)^*,\quad \mu\leq 0,
\end{align}
which will later be pressure and dissipation measure of the subsolution considered in Theorem \ref{thm:main_theorem}. We will then look for $L^\infty$-functions 
\[
z:=(v,m,\sigma,e):\T^2\times(0,T)\to \R^2\times\R^2\times\cS_0^{2\times 2}\times\R=:Z,
\]
such that $z$ is a solution of \eqref{eq:linear_system_with_p} in the sense of Definition \ref{def:weaksolslin}, the induced pressure is $p$ and there holds 
\begin{equation}\label{eq:nonlinear_constraints}
z(x,t)\in K_{(x,t)}:=\set{z\in Z: v\otimes v-\sigma =e\id,\quad m=(e+ p(x,t))v},
\end{equation}
for almost every $(x,t)\in \T^2\times (0,T).$ Indeed, under these conditions $v$ is a bounded globally dissipative solution of the Euler equations in the sense of Definition \ref{def:weaksolseu} having the fixed pressure $p$ and the fixed energy dissipation $\mu$.

The general strategy of convex integration in the Tartar framework relies on the idea that if one can find a weak solution $\hat z$ of \eqref{eq:linear_system_with_p} which instead takes values in the ($\Lambda$-)convex hull, i.e. $\hat{z}(x,t)\in\text{int}(K_{(x,t)}^{co})$, then one may deduce the existence of infinitely many solutions $z$, which are near $\hat{z}$ in the weak sense while satisfying $z(x,t)\in K_{(x,t)}$ a.e., by adding some specially constructed perturbations to $\hat{z}$.

As perturbations we will use plane-wave like solutions to \eqref{eq:linear_system_with_p} which actually do not perturb the pressure or the dissipation rate. That is we consider the system
\begin{align}\label{eq:linear_system}
\begin{split}
\partial_t\bar v+\divv \bar\sigma +\nabla \bar e&=0,\\
\divv \bar v&=0,\\
\partial_t \bar e+\divv \bar m &=0.
\end{split}
\end{align}
Note that if $z$ is a weak solution of \eqref{eq:linear_system_with_p} and $\bar z\in\cC_c^\infty(\T^2\times (0,T))$ solves \eqref{eq:linear_system}, then $z+\bar{z}$ is also a weak solution of \eqref{eq:linear_system_with_p}, with the same pressure $p$ and energy dissipation rate $\mu$. Also the initial data remains unchanged.

\subsection{Localized plane waves}\label{sec:waves}

The wave cone associated with \eqref{eq:linear_system} reads:
\begin{equation}\label{eq:wave_cone}
\Lambda=\set{\bar{z}\in Z:\ker \begin{pmatrix}
\bar{\sigma}+\bar{e}\id & \bar{v}\\
\bar{v}^T & 0\\
\bar{m}^T & \bar{e}
\end{pmatrix} \neq \{0\},\quad (\bar{v},\bar{e})\neq0}.
\end{equation}
Note that for $\bar{z}\in\Lambda$ there exists $\eta\in\R^{3}\setminus\{0\}$ such that every $z(x,t)=\bar{z}h((x,t)\cdot\eta)$, $h\in\cC^1(\R)$ is a solution of \eqref{eq:linear_system}. This allows us to construct solutions which oscillate in the direction $\bar z$. Observe that the condition $(\bar{v},\bar{e})\neq0$ serves to eliminate the degenerate case when the first two components of $\eta$ vanish, i.e. when one would be allowed to oscillate in time only. 

In Lemma \ref{lem:locpw} below we construct localized plane wave-like solutions to \eqref{eq:linear_system}. Here $d$ denotes the Euclidean distance on $Z$.

\begin{lemma}\label{lem:locpw}
There exists $C_0>0$ such that for any $\bar{z}\in\Lambda$, there exists a sequence
$z_N\in \cC_c^\infty(B_1(0);Z)$, $B_1(0)$ being the unit ball in $\R^2\times\R$, 
solving \eqref{eq:linear_system} and satisfying 
\begin{itemize}
\item[(i)] $d(z_N,[-\bar{z},\bar{z}])\to 0$ uniformly,
\item[(ii)] $z_N\rightharpoonup 0$ in $L^2(B_1(0);Z)$,
\item[(iii)] $\int\int |z_N|^2\, dx \, dt\geq C_0|\bar{z}|^2.$
\end{itemize}
\end{lemma}
\begin{proof} For $x\in\R^2$ denote $x^\perp:=(-x_2,x_1)^T$.
Let $\bar{z}\in\Lambda$. 
By definition,
there exists 
\begin{align}\label{eq:xc}
0\neq(\xi,c)\in\ker \begin{pmatrix}
\bar{\sigma}+\bar{e}\id & \bar{v}\\
\bar{v}^T & 0\\
\bar{m}^T & \bar{e}
\end{pmatrix}.
\end{align}
Furthermore, $\xi\neq 0$, because $\xi=0$ would imply $(\bar{v},\bar{e})=0$. Without loss of generality we will assume that $|\xi|=1$.

Recall from \cite[Remark 2]{DeL-Sz-Annals} that for any smooth function $\omega:\mathbb{R}^{2+1}\to\mathbb{R}^{2+1}$, defining $W:=\curl_{(x,t)}\omega$ and
\begin{align*}
v:=-\frac{1}{2}\nabla^\perp W_3,\quad
\sigma+e\id:=\begin{pmatrix}
\partial_2 W_1 & \frac{1}{2}(\partial_2W_2-\partial_1W_1)\\
\frac{1}{2}(\partial_2W_2-\partial_1W_1) & -\partial_1W_2
\end{pmatrix}
\end{align*}
implies that 
\begin{align*}
\partial_t v+\divv \sigma +\nabla e&=0,\\
\divv v&=0.
\end{align*}

We observe that $e=\frac{1}{2}\divv (-W_2,W_1)$, so if we define $m:=-\frac{1}{2}\partial_t (-W_2,W_1)$, then we have that $D(\omega):=(v,m,\sigma,e)$ solves \eqref{eq:linear_system}.

On the other hand, let $S:\mathbb{R}\to\mathbb{R}$ be a smooth function, $N\geq 1$, and let us define the constants $A,B,C$ as follows:
$$C:=-2|\bar v|\text{sgn}(\xi^\perp\cdot \bar v),\quad (A,B)^T=-cC\xi-2\bar{e}\xi^\perp.$$
If we then consider $W^N$ of the form
$$W^N(x,t)=(A,B,C)\frac{1}{N} S'(N (\xi,c)\cdot (x,t)),$$
it is easy to check that since $(A,B,C)\cdot(\xi,c)=0,$ it follows that $\divv_{(x,t)}W^N=0$, so there exists $\omega_N$ such that $W^N=\curl_{(x,t)}\omega_N.$
Furthermore, simple calculation yields
$$(v,\sigma,e)(x,t)=(\bar v,\bar \sigma,\bar e)S''(N (\xi,c)\cdot (x,t)),$$
as well as
$$m=-\frac{1}{2}\partial_t (-W^N_2,W^N_1)=-\frac{1}{2}c(-B,A) S''(N (\xi,c)\cdot (x,t)).$$
Using $\bar{m}\cdot \xi+c \bar e=0$, we further obtain $-\frac{1}{2}c(-B,A)=\frac{c^2C}{2}\xi^\perp+(\bar{m}\cdot \xi) \xi$, so it follows that
$$D(\omega_N)=\left(\bar{v},\bar{m}+\left(\frac{c^2C}{2}-\bar{m}\cdot \xi^\perp\right) \xi^\perp,\bar{\sigma},\bar{e}\right) S''(N (\xi,c)\cdot (x,t)).$$

However, observe that for any smooth real valued function $\theta:\mathbb{R}^{2+1}\to\mathbb{R}$, $\hat{D}(\theta)=(0,\nabla^\perp\theta,0,0)$ also solves \eqref{eq:linear_system}. 
Therefore, we may consider the potential given by
$$\theta_N(x,t)=-\left(\frac{c^2C}{2}-\bar{m}\cdot \xi^\perp\right) \frac{1}{N}S'(N (\xi,c)\cdot (x,t)),$$
to obtain that
$$\nabla^\perp\theta_N(x,t)=-\left(\frac{c^2C}{2}-\bar{m}\cdot \xi^\perp\right) \xi^\perp S''(N (\xi,c)\cdot (x,t)),$$
and in conclusion,
\begin{align*}
D(\omega_{N})+\hat{D}(\theta_N)=\bar{z} S''(N (\xi,c)\cdot (x,t)).
\end{align*}
In order to conclude the proof of the lemma, it remains to localize this potential in the usual way (e.g. as in \cite{Cordoba,DeL-Sz-Annals}).
One may fix $S(\cdot)=-\cos(\cdot)$ and, for $\varepsilon>0$, consider $\chi_\varepsilon\in \cC_c^\infty(B_1(0))$ satisfying $|\chi_\varepsilon|\leq 1$ on $B_{1}(0)$, $\chi_\varepsilon=1$ on $B_{1-\varepsilon}(0)$. One can then check through simple calculations that 
$z_N=D(\chi_\varepsilon\omega_{N})+\hat{D}(\chi_\varepsilon\theta_N)$ satisfies the conclusions of the lemma.
\end{proof}

\begin{remark}\label{rem:plane_waves_at_time_0}
For later purposes we like to state that there also exists a constant $C_0>0$ independent of $\bar{z}$, such that the constructed sequence $z_N$ at the time slice $t=0$ satisfies $z_N(\cdot,0)\rightharpoonup 0$ in $L^2(B_1(0);Z)$ and $\int_{B_1(0)}z_N(x,0)\:dx\geq C_0\abs{\bar{z}}^2$, where now $B_1(0)\subset\R^2$ is the unit ball in space. Here it is important that $\xi\neq 0$. Moreover, by setting
$z_N=D(\chi_\varepsilon^N\omega_{2^N})+\hat{D}(\chi^N_\varepsilon\theta_{2^N})$ with a suitable cutoff function $\chi_\varepsilon^N$ satisfying $\supp \chi_\varepsilon^N\subset B_1(0)\times\left(-\frac{c}{N},\frac{c}{N}\right)$, $c>0$ we can in addition make sure that the support of $z_N$ is contained in $B_1(0)\times\left(-\frac{c}{N},\frac{c}{N}\right)$.
\end{remark}

\subsection{Perturbing along sufficiently long segments}\label{sec:segments}

In this subsection we prove that the wave cone $\Lambda$ is large with respect to $K_{(x,t)}$, in the sense that any two points in $K_{(x,t)}$ can be connected with a $\Lambda$-segment. Furthermore, this property automatically implies that any point in the interior of the convex hull of any compact subset of $K_{(x,t)}$ can be perturbed along sufficiently long $\Lambda$-segments.

For simplicity of notation, for the rest of the subsection we will fix a point $(x,t)\in\T^2\times (0,T)$ and write $K$ instead of $K_{(x,t)}$.

 
 \begin{lemma}\label{lem:big_fkn_cone}
 For any $z_1,z_2\in K$, $z_1\neq z_2$, we have $\bar{z}=z_2-z_1\in\Lambda$.
 \end{lemma}
 \begin{proof}
 Since $z_i\in K$, we have $z_i=(v_i,(e_i+p)v_i,v_i\otimes v_i-e_i\id ,e_i)$, $e_i=\frac{1}{2}\abs{v_i}^2$, $i=1,2$, and therefore $\bar{v}\neq 0$.
  
 If $\bar{e}=0$, then all that needs to be checked is that $\bar{\sigma}\bar{v}^\perp=c\bar{v}$, for some $c\in\R$. There holds
 \begin{align*}
 \bar{\sigma}\bar{v}^\perp=(v_2\otimes v_2-v_1\otimes v_1)\bar{v}^\perp=(v_2\otimes \bar{v}+\bar{v}\otimes v_1)\bar{v}^\perp=(v_1\cdot\bar{v}^\perp)\bar{v},
 \end{align*}
 so $\bar{z}\in\Lambda$ follows.
 
  If $\bar{e}\neq 0$, we similarly obtain from  $z_1,z_2\in K$ that
  $$(\bar{\sigma}+\bar{e}\id)\bar{v}^\perp=(v_2\otimes v_2-v_1\otimes v_1)\bar{v}^\perp=(v_1\cdot\bar{v}^\perp)\bar{v},$$
  so it remains to check that $\bar{m}\cdot\bar{v}^\perp=(v_1\cdot\bar{v}^\perp)\bar{e}$ also holds. We have
  $$\bar{m}\cdot\bar{v}^\perp=(e_2v_2-e_1v_1+p\bar{v})\cdot\bar{v}^\perp=(e_2\bar{v}+\bar{e}v_1)\cdot\bar{v}^\perp=(v_1\cdot\bar{v}^\perp)\bar{e},$$
  the result then follows.
 \end{proof}
 Solely by this property we have the following result.
 \begin{corollary}\label{cor:seg}
 Let $K'\subset K$ be a compact set. For any $z\in\text{int} (K')^{co}$ there exists $\bar{z}\in \Lambda$ such that
 $$[z-\bar z,z+\bar z]\subset \text{int} (K')^{co}\text{ and }|\bar z|\geq\frac{1}{2N}d(z,K'),$$
 where $N=\text{dim}(Z)$ and $d$ is the Euclidean distance on $Z$.
 \end{corollary}
 The proof is the same as those of Lemma 6 from \cite[Lemma 6]{DeL-Sz-Adm}, respectively \cite[Lemma 4.9]{GKSz}, relying on Carath\'eodory's theorem and Lemma \ref{lem:big_fkn_cone} above, therefore we omit it.

\subsection{The convex hull}\label{sec:convex_hull_unconstraint}

We now explicitly compute the full $\Lambda$-convex hull associated with our differential inclusion, which turns out to coincide with the usual convex hull. Moreover, Proposition \ref{prop:hull} shows that relaxing the energy inequality does not give a new condition in the hull compared to the case of the Euler equations with prescribed energy function known from \cite{DeL-Sz-Adm}. 
The definition of the $\Lambda$-convex hull $(K')^\Lambda$ of $K'\subset Z$ can be recalled for example from \cite{Kirchheim,Matousek_1998}: we say that $z\in (K')^\Lambda$ if and only if, for all $\Lambda$-convex functions $f:Z\rightarrow\R$, there holds $f(z)\leq \sup_{z'\in K'}f(z')$.

Recall also that $(K')^\Lambda$ is closed provided that the linear span of $\Lambda$ is all of $Z$, \cite[Corollary 2.4]{Matousek_1998}. In our case the latter property follows from $(\bar{v},0,0,0)\in \Lambda$, $(\bar{v},\bar{v},0,0)\in\Lambda$ and $(0,0,\bar{\sigma},\pm\lamax(\bar{\sigma}))\in \Lambda$ for any $\bar{v}\neq 0$, $\bar{\sigma}\neq 0$. 
\begin{proposition}\label{prop:hull}
Independently of $p$ and $(x,t)$ there holds
\begin{equation}\label{eq:computation_of_full_hull}
K_{(x,t)}^\Lambda=K_{(x,t)}^{co}=\set{z\in Z:\lambda_{max}(v\otimes v-\sigma)\leq e}=:\overline{U}.
\end{equation}
\end{proposition}
\begin{proof}
It is clear that the right-hand side is a convex set, cf. \cite[Lemma 3(i)]{DeL-Sz-Adm}, containing $K_{(x,t)}$. Hence $K_{(x,t)}^\Lambda\subset K_{(x,t)}^{co}\subset\overline{U}$. For the other inclusion observe that we can not use the Krein-Milman theorem for $\Lambda$-convex sets, \cite[Lemma 4.16]{Kirchheim}, since $K_{(x,t)}$ is not compact. We instead conclude the statement by direct computation of $\Lambda$-segments, which is carried out in Lemmas \ref{lem:segments_of_full_hull_1}, \ref{lem:segments_of_full_hull_2} below. Indeed taking  then the closure of the set specified in Lemma \ref{lem:segments_of_full_hull_2} shows that $\overline{U}\subset K_{(x,t)}^\Lambda$.
\end{proof}
For the rest of this subsection we will drop the $(x,t)$ dependence in the notation. Let 
\[
K^{\Lambda,1}:=K\cup \set{sz_1+(1-s)z_2:z_1,z_2\in K,~s\in[0,1],~z_1-z_2\in\Lambda}
\]
and  $K^{\Lambda,j}:=\left(K^{\Lambda,j-1}\right)^{\Lambda,1}$, $j\geq 2$. By the definition of $K^\Lambda$ one clearly has $K^{\Lambda,j}\subset K^\Lambda$ for any $j\geq 1$. 

Furthermore, we define for $z\in Z$ with $m\neq (e+p)v$ the vector
\[
\eta(z):=\frac{m-(e+p)v}{\abs{m-(e+p)v}},
\]
as well as the traceless matrix
\[
M(z):=v\otimes v-\sigma -e\id +(2e-\abs{v}^2)\eta(z)\otimes\eta(z).
\]
\begin{lemma}\label{lem:segments_of_full_hull_1}
Every $z\in Z$ satisfying $m\neq (e+p)v$, $2e-\abs{v}^2>0$, $M(z)=0$ is an element of $K^{\Lambda,1}$.
\end{lemma}
\begin{proof}
Let $z$ be as stated and define 
\begin{gather}
\begin{gathered}\label{eq:definition_of_first_Lambda_direction}
\bar{e}=1,\quad \bar{v}=\frac{2e-\abs{v}^2}{\abs{m-(e+p)v}}\eta(z), \quad \beta=\frac{1-v\cdot\bar{v}}{\abs{\bar{v}}^2}\\
\bar{m}=v+\bar{v}(e+p+2\beta),\quad\bar{\sigma}=\big[v\otimes \bar{v}+\bar{v}\otimes v+2\beta\bar{v}\otimes \bar{v}\big]^\circ,
\end{gathered}
\end{gather}
where $[S]^\circ$ is the trace-free part of $S\in \R^{2\times 2}$. We will show that there exists $s_1<0<s_2$, such that $z+s_{1,2}\bar{z}\in K$. In consequence Lemma \ref{lem:big_fkn_cone} implies $\bar{z}\in \Lambda$ and thus $z\in K^{\Lambda,1}$.

Now for $z+s\bar{z}\in K$ we on one hand need to satisfy 
\begin{align*}
0=m+s\bar{m}-(e+p+s)(v+s\bar{v})=m-(e+p)v+s(\bar{m}-v-(e+p)\bar{v})-s^2\bar{v},
\end{align*}
which by the definition of $\bar{v},\bar{m}$ is equivalent to
\begin{align}\label{eq:equation_for_roots_of_first_order_segments}
s^2-2\beta s-\frac{\abs{m-(e+p)v}^2}{2e-\abs{v}^2}=0.
\end{align}
Since $2e-\abs{v}^2>0$ by assumption, this equation has two solutions $s_1,s_2$ with different signs. For $s=s_{1,2}$ we on the other hand also need to satisfy 
\begin{align}
\begin{split}\label{eq:vanishing_of_matrix_in_1st_segments}
0&=(v+s\bar{v})\otimes(v+s\bar{v})-(\sigma+s\bar{\sigma})-(e+s)\id\\
&=s^2\bar{v}\otimes\bar{v}+s(\bar{v}\otimes v+v\otimes \bar{v}-\bar{\sigma}-\id)+v\otimes v-\sigma-e\id\\
&=\left(s^2-2\beta s -\frac{\abs{m-(e+p)v}^2}{2e-\abs{v}^2}\right)\bar{v}\otimes\bar{v},
\end{split}
\end{align}
which holds due to \eqref{eq:equation_for_roots_of_first_order_segments}. In the last step we used the definition of $\bar{v}$, $\bar{\sigma}$ and $\beta$, as well as the assumption $M(z)=0$.
\end{proof}

\begin{lemma}\label{lem:segments_of_full_hull_2}
If $z\in Z$ satisfies $m\neq (e+p)v$ and $\lambda_{\max}(v\otimes v-\sigma)<e$, then $z\in K^{\Lambda,2}$.
\end{lemma}
\begin{proof}
We will show that every such $z$ is contained in a $\Lambda$-segment with endpoints in the subset of $K^{\Lambda,1}$ given in Lemma \ref{lem:segments_of_full_hull_1}. 

We consider the matrices
\[
Y_1(z):=\eta(z)\otimes\eta(z)-\eta(z)^\perp\otimes\eta(z)^\perp,\quad Y_2(z):=\eta(z)\otimes\eta(z)^\perp+\eta(z)^\perp\otimes\eta(z)
\]
as a basis of $\cS_0^{2\times 2}$ and denote by $M_1(z),M_2(z)$ the components of $M(z)$ with respect to this basis. I.e., $M(z)=M_1(z)Y_1(z)+M_2(z)Y_2(z)$.
Observe that 
\begin{align}
\begin{split}\label{eq:M_1_positive}
M_1(z)&=\eta(z)^TM(z)\eta(z)=e-\abs{v}^2+\eta(z)^T(v\otimes v-\sigma)\eta(z)\\
&\geq e-\abs{v}^2+\lamin(v\otimes v-\sigma)=e-\lamax(v\otimes v-\sigma)>0.
\end{split}
\end{align}
For later use we also state the corresponding upper bound
\begin{equation}\label{eq:M_1_bounded_above}
M_1(z)\leq e-\lamin(v\otimes v-\sigma).
\end{equation}
The $\Lambda$-direction that serves our purpose will be
\begin{align}\label{eq:definition_of_second_Lambda_segments}
\begin{gathered}
\bar{e}=0,\quad \bar{v}=-\frac{M_2(z)}{\abs{M(z)}}\eta(z)+\frac{M_1(z)}{\abs{M(z)}}\eta(z)^\perp,\quad \bar{m}=(e+p)\bar{v},\\
\bar{\sigma}=v\otimes\bar{v}+\bar{v}\otimes v+2v\cdot\bar{v}\left(\frac{M_1(z)}{\abs{M(z)}^2}M(z)-\eta(z)\otimes\eta(z)\right),
\end{gathered}
\end{align}
where $\abs{M(z)}^2=M_1(z)^2+M_2(z)^2>0$. Next we check that $\bar{z}$ defined in this way is indeed an element of $\Lambda$. Recall that $\tr M(z)=0$, such that $\bar{\sigma}\in \cS^{2\times 2}_0$.
Choose $\xi=\bar{v}^\perp=\abs{M(z)}^{-1}\big(-M_2(z)\eta(z)^\perp-M_1(z)\eta(z)\big)$. Since $\bar{e}=0$ and $\bar{m}\in\R\bar{v}$ it only remains to find $c\in\R$, such that $\bar{\sigma}\xi+c\bar{v}=0$. This is possible, since 
\begin{align*}
\xi^T\bar{\sigma}\xi&=2v\cdot\bar{v}\left(\frac{M_1(z)}{\abs{M(z)}^2}\xi^TM(z)\xi-(\eta(z)\cdot\xi)^2\right)\\
&=2v\cdot\bar{v}\left(\frac{M_1(z)}{\abs{M(z)}^2}M_1(z)-\frac{M_1(z)^2}{\abs{M(z)}^2}\right)=0.
\end{align*}
Using now this $\Lambda$ direction we directly see that 
\[
m+s\bar{m}-(e+s\bar{e}+p)(v+s\bar{v})=m-(e+p)v
\]
for all $s\in\R$. Hence $\eta(z+s\bar{z})=\eta(z)$ for any $s\in\R$ and it remains to find $s_1<0<s_2$ with $2e-\abs{v+s\bar{v}}^2>0$ and $M(z+s\bar{z})=0$. There holds
\begin{align}\label{eq:roots_of_second_segments}
\begin{split}
M(z+s\bar{v})&=s^2(\bar{v}\otimes\bar{v}-\eta(z)\otimes\eta(z))\\
&\hspace{20pt}+s(v\otimes\bar{v}+\bar{v}\otimes v-2v\cdot\bar{v}\eta(z)\otimes\eta(z)-\bar{\sigma})+M(z)\\
&=\left(-\frac{M_1(z)}{\abs{M(z)}^2}s^2-2v\cdot\bar{v}\frac{M_1(z)}{\abs{M(z)}^2}s+1\right)M(z).
\end{split}
\end{align}
By \eqref{eq:M_1_positive} there exist $s_1<0<s_2$ solving this quadratic equation. Now all that needs to be checked for $s=s_{1,2}$ is the inequality $2e-\abs{v+s\bar{v}}^2>0$, which by the definition of $s_{1,2}$ can be rewritten as
\begin{align}
\label{eq:condition_for_second_segments}
0<(2e-\abs{v}^2)M_1(z)-\abs{M(z)}^2.
\end{align}
We abbreviate 
\begin{gather*}
a_1=\eta(z)^T(v\otimes v-\sigma)\eta(z),\quad a_2=\eta(z)^T(v\otimes v-\sigma)\eta(z)^\perp,\\
a_3=(\eta(z)^\perp)^T(v\otimes v-\sigma)\eta(z)^\perp.
\end{gather*}
Using $\abs{v}^2=\tr(v\otimes v-\sigma)=a_1+a_3$ and $a_1a_3-a_2^2=\det(v\otimes v-\sigma)$ the right-hand side of \eqref{eq:condition_for_second_segments} becomes
\begin{align}
\begin{split}\label{eq:M_v_identity}
(2e-\abs{v}^2)M_1(z)-\abs{M(z)}^2&=(2e-\abs{v}^2)(e-\abs{v}^2+a_1)-(e-\abs{v}^2+a_1)^2-a_2^2\\
&=e^2-e\abs{v}^2+\abs{v}^2a_1-a_1^2-a_2^2\\
&=\left(e-\frac{1}{2}\abs{v}^2\right)^2-\left(\frac{1}{4}\abs{v}^4-\det(v\otimes v-\sigma)\right)\\
&=\left(e-\frac{1}{2}\abs{v}^2\right)^2-\left(\lamax(v\otimes v-\sigma)-\frac{1}{2}\abs{v}^2\right)^2\\
&=\big(e-\lamax(v\otimes v-\sigma)\big)\big(e-\lamin(v\otimes v-\sigma)\big),
\end{split}
\end{align}
which is positive by assumption.
\end{proof}

\begin{remark}\label{rem:full_hull}
If we also consider the pressure $p$ as part of the variables and not as a given function,  Proposition \ref{prop:hull} implies that the $\Lambda'$-convex hull of the set
\[
\set{(v,m,\sigma,e,p)\in\R^2\times \R^2\times\cS_0^{2\times 2}\times\R\times \R:v\otimes v-\sigma =e\id,~m=(e+p)v} 
\]
is given by the convex unbounded set
\[
\set{(v,m,\sigma,e,p)\in\R^2\times \R^2\times\cS_0^{2\times 2}\times\R\times \R:\lamax(v\otimes v-\sigma)\leq e}.
\]
This statement holds true for $\Lambda'$ defined as the directions $(\bar{z},0)$ with $\bar{z}\in\Lambda$ or any bigger wave cone, e.g. the wave cone associated with \eqref{eq:linear_system_with_p} when $p$ is not fixed, or the full cone $\Lambda'=Z\times \R$.
\end{remark}

\subsection{\texorpdfstring{An $L^\infty$ bound}{Introducing a bound}}\label{sec:infty}
Now that we know the full $\Lambda$-convex hull, we essentially know the complete relaxation of our differential inclusion. However, in order to apply the usual convex integration in the Tartar framework, which relies on a bounded subset of $L^2(\T^2\times(0,T);Z)$ functions, we introduce an $L^\infty$-bound on $e$ in the set of nonlinear constraints. We fix a constant $\gamma>0$ and set 


\begin{equation}\label{eq:nonlinear_constraints_with_gamma}
K_{\gamma,(x,t)}:=\set{z\in K_{(x,t)}:e\leq \gamma}.
\end{equation}
One easily sees that any $z\in K_{\gamma,(x,t)}$ is bounded in terms of $\gamma$ and $p(x,t)$, cf. \cite[Lemma 3(iii)]{DeL-Sz-Adm} for the bound on $\sigma$.

Note that introducing such a bound is not a restriction in the case of sufficiently smooth initial data $v_0$, because then the induced velocity field $v$ is uniformly bounded on $\T^2\times[0,T)$, see for instance \cite{Marchioro_Pulvirenti}. In other words there exists a fixed $\gamma>0$, such that the corresponding function $z=(v,(\abs{v}^2/2+p)v,(v\otimes v)^\circ,\abs{v}^2/2)$ solves the linear system \eqref{eq:linear_system_with_p} with $\mu\equiv 0$ and satisfies $z(x,t)\in K_{\gamma,(x,t)}$ for all $(x,t)\in \T^2\times[0,T)$. 

From now on we again drop the $(x,t)$ dependence in our notation.

\begin{lemma}\label{lem:segments_of_gamma_hull_1}
Let $z\in Z$ satisfy $m\neq(e+p)v$, $2e-\abs{v}^2>0$, $M(z)=0$ and in addition $e<\gamma$,
\begin{equation}\label{eq:additional_condition_for_1st_segments}
\abs{\frac{(m-(e+p)v)(2\gamma-\abs{v}^2)}{(\gamma-e)(2e-\abs{v}^2)}-v}\leq\sqrt{2\gamma}.
\end{equation}
Then $z\in K^{\Lambda,1}_\gamma$.
\end{lemma}
\begin{proof}
From the proof of Lemma \ref{lem:segments_of_full_hull_1} it only remains to verify the new condition $e+s_{1,2}\leq \gamma$ for $s_1<0<s_2$ defined as the roots of \eqref{eq:equation_for_roots_of_first_order_segments} in order to conclude that the endpoints of the $\Lambda$-segment $[z+s_1\bar{z},z+s_2\bar{z}]$ belong to $K_\gamma$. The $\bar{z}\in\Lambda$ is the one defined in \eqref{eq:definition_of_first_Lambda_direction}. Since $s_1<s_2$, it is in fact enough to show $e+s_2\leq \gamma$.

Let $f(s):=\abs{v+s\bar{v}}^2-2(e+s)$ be the trace of the right-hand side of equation \eqref{eq:vanishing_of_matrix_in_1st_segments} and observe that by construction we have $f(s)=0$ exactly for $s=s_{1,2}$. The condition $s_2\leq \gamma-e$ therefore is equivalent to $\gamma-e>0$, $f(\gamma-e)\geq 0$. Now $e<\gamma$ by assumption and 
\begin{align*}
0\leq f(\gamma-e)&=\abs{v+(\gamma-e)\frac{2e-\abs{v}^2}{\abs{m-(e+p)v}}\eta(z)}^2-2\gamma
\end{align*}
is equivalent to 
\begin{align*}
0\leq \frac{\abs{m-(e+p)v}^2(\abs{v}^2-2\gamma)}{(\gamma-e)^2(2e-\abs{v}^2)^2}+2\frac{v\cdot(m-(e+p)v)}{(\gamma-e)(2e-\abs{v}^2)}+1.
\end{align*}
Multiplying the last inequality with $\abs{v}^2-2\gamma<0$ we obtain \eqref{eq:additional_condition_for_1st_segments}.
\end{proof}


At this point we stop computing the hull in an explicit way and instead provide an easily readable condition under which an element belongs to $K^\Lambda_\gamma$ for a suitably chosen $\gamma$. 
\begin{proposition}\label{prop:condition_for_being_in_K_gamma_hull}
Assume that $z\in Z$ satisfies 
\begin{align}\label{eq:suff_condition_in_proposition}
\lamax(v\otimes v-\sigma)+\varepsilon\abs{m-(e+p)v}< e
\end{align}
for some $\varepsilon>0$. Then there exists $\gamma=\gamma_\varepsilon(e)>0$ depending only on $\varepsilon$ and $e$, such that $z\in \text{int}\left(K_\gamma^\Lambda\right)$. Moreover, the map $\R_+\ni e\mapsto\gamma_\varepsilon(e)\in\R_+$ is continuous. 
\end{proposition}
\begin{proof}
Let $\tilde{z}=(\tilde{v},\tilde{m},\tilde{\sigma},\tilde{e})\in Z$, $\varepsilon>0$ be satisfying \eqref{eq:suff_condition_in_proposition}. By continuity we can find $\delta\in (0,\tilde{e})$, such that \eqref{eq:suff_condition_in_proposition} holds true for all $z\in B_\delta(\tilde{z})$. We will show that any such $z$ is contained in $K_{\gamma}^\Lambda$ with a $\gamma>0$ only depending on $\varepsilon$ and $\tilde{e}$. 

Let $z\in B_\delta(\tilde{z})$. Assume first that $m=(e+p)v$. It follows from the known convex hull of the usual Euler equation, cf. \cite[Lemma 3(iv)]{DeL-Sz-Adm},
that $z\in K_\gamma^{\Lambda'}\subset K_\gamma^\Lambda$ for any $\gamma> e$ and $\Lambda':=\set{\bar{z}\in\Lambda:\bar{e}=0,~\bar{m}=(e+p)\bar{v}}$.

Assume next $m\neq (e+p)v$. In the proof of Lemma \ref{lem:segments_of_full_hull_2} we showed that such a $z$ is lying on the segment $[z+s_1\bar{z},z+s_2\bar{z}]$ with $\bar{z}\in\Lambda$ defined in \eqref{eq:definition_of_second_Lambda_segments} and $s_1<0<s_2$ defined as the roots of \eqref{eq:roots_of_second_segments}, such that $z+s_{1,2}\bar{z}\in K^{\Lambda,1}$ and hence $z\in K^{\Lambda,2}$. In order to have now $z+s_{1,2}\bar{z}\in K_\gamma^{\Lambda,1}$ we, because of Lemma \ref{lem:segments_of_gamma_hull_1}, need to make sure that $\gamma>e+s_{1,2}\bar{e}=e$ and that \eqref{eq:additional_condition_for_1st_segments} holds for the two points $z+s_{1,2}\bar{z}$.

Recall that along the direction $\bar{z}$ the difference $m-(e+p)v$ is conserved and also that 
\[
\abs{v+s_{1,2}\bar{v}}^2=\abs{v}^2+\frac{\abs{M(z)}^2}{M_1(z)}
\]
by the definition of $s_{1,2}$ as the roots of \eqref{eq:roots_of_second_segments}. By \eqref{eq:M_1_bounded_above} and \eqref{eq:M_v_identity} we therefore obtain
\[
2e-\abs{v+s_{1,2}\bar{v}}^2=\frac{(e-\lamax(v\otimes v-\sigma))(e-\lamin(v\otimes v-\sigma))}{M_1(z)}\geq e-\lamax(v\otimes v-\sigma).
\]
Plugging $z+s_{1,2}\bar{z}$ into the left-hand side of \eqref{eq:additional_condition_for_1st_segments} and using the latter inequality we estimate
\begin{align*}
&\abs{\frac{(m-(e+p)v)(2\gamma-\abs{v+s_{1,2}\bar{v}}^2)}{(\gamma-e)(2e-\abs{v+s_{1,2}\bar{v}}^2)}-(v+s_{1,2}\bar{v})}\\
&\hspace{60pt}\leq \frac{2\gamma\abs{m-(e+p)v}}{(\gamma-e)(e-\lamax(v\otimes v-\sigma))}+\frac{1}{2}\abs{v+s_{1,2}\bar{v}}^2+\frac{1}{2}\\
&\hspace{60pt}< \frac{2\gamma}{\varepsilon(\gamma-e)}+e+\frac{1}{2}.
\end{align*}
In order to satisfy inequality \eqref{eq:additional_condition_for_1st_segments} for $z+s_{1,2}\bar{z}$, it therefore is enough to choose $\gamma=\max\set{2e,(e+1/2+4/\varepsilon)^2/2}$, such that
\[
\frac{2\gamma}{\varepsilon(\gamma-e)}+e+\frac{1}{2}\leq \sqrt{2\gamma}.
\]

Note that this choice is still dependent (via $e$) on the considered $z\in B_\delta(\tilde{z})$. In order to make a choice depending only on $\varepsilon$ and $\tilde{e}$ we recall that $\delta<\tilde{e}$ and simply set
\[
\gamma_\varepsilon(\tilde{e}):=\max\set{4\tilde{e},(2\tilde{e}+1/2+4/\varepsilon)^2/2}.
\]
\end{proof}

\subsection{Conclusion}\label{sec:conclusion}
We begin with the following more abstract version of Theorem \ref{thm:main_theorem}.
\begin{theorem}\label{thm:abstract_CI_theorem}
Let $\gamma>0$. Theorem \ref{thm:main_theorem} remains true when condition \eqref{eq:sufficient_condition_for_Linfty_hull} is replaced by
\begin{align}\label{eq:abstract_condition_for_hull}
    (v(x,t),m(x,t),\sigma(x,t),e(x,t))\in \text{int} \left( K^{co}_{\gamma,(x,t)}\right)\text{ for all }(x,t)\in\mathscr{U}.
\end{align}
\end{theorem}
Here we mean that Theorem \ref{thm:main_theorem} holds true word by word except that \eqref{eq:sufficient_condition_for_Linfty_hull} is swapped with \eqref{eq:abstract_condition_for_hull}. For further clarification we point out that the set $K_{\gamma,(x,t)}$ is defined via \eqref{eq:nonlinear_constraints}, \eqref{eq:nonlinear_constraints_with_gamma} with respect to the pressure $p$ induced by $(v,m,\sigma,e)$, cf. Definition \ref{def:weaksolslin}.
\begin{proof}[Proof of Theorem \ref{thm:abstract_CI_theorem}]
We start by observing that if $p$ was constant in $\mathscr{U}$, then \eqref{eq:linear_system_with_p} together with the set of constraints $K_\gamma$ would  fit into the framework stated in the appendix of \cite{Sz-Muskat}. In fact, due to Sections \ref{sec:waves}, \ref{sec:segments} and \ref{sec:infty} above, the result would directly follow  from \cite[Theorem 5.1]{Sz-Muskat}.

However, the Tartar framework can easily be adapted to the case when the set of constraints also depends on $(x,t)$. The extra condition which is needed is that the map
$(x,t)\mapsto K_{\gamma,(x,t)}$ is continuous and bounded on $\mathscr{U}$ with respect to the Hausdorff metric $d_{\mathcal{H}}$, see \cite{Crippa}. We prove this in Lemma \ref{lem:cont_of_constr} below. 
This allows us to conclude the proof of Theorem \ref{thm:abstract_CI_theorem}.
\end{proof}

\begin{lemma}\label{lem:cont_of_constr}
Let $\mathscr{U}\subset\T^2\times(0,T)$ be an open set and $p:\mathscr{U}\rightarrow\R$ be a continuous and bounded function. The map $(x,t)\mapsto K_{\gamma,(x,t)}$ is continuous and bounded on $\mathscr{U}$ with respect to the Hausdorff metric $d_{\mathcal{H}}$.
\end{lemma}
\begin{proof}
Let $y=(x,t)\in\mathscr{U}$. For $\varepsilon>0$ there exists $\delta>0$ such that 
\begin{align}\label{eq:squareroot}
\abs{p(y)-p(y')}<\frac{1}{\sqrt{2\gamma}}\varepsilon,
\end{align} for any $y'\in B_\delta(y)\subset\mathscr{U}.$
Regarding the continuity, due to \cite[Lemma 3.1]{Crippa}, it suffices to prove that 
\begin{itemize}
\item for any $z\in K_{\gamma,y}$ there exists $z'\in K_{\gamma,y'}\cap B_{\varepsilon}(z),$
\item for any $z\in K_{\gamma,y'}$ there exists $z'\in K_{\gamma,y}\cap B_{\varepsilon}(z),$
\end{itemize}
since then $d_{\mathcal{H}}(K_{\gamma,y},K_{\gamma,y'})< \varepsilon$. Therefore let $$z=(v,(e+p(y))v,v\otimes v-e\id,e)\in K_{\gamma,y},$$
and define $$z'=(v,(e+p(y'))v,v\otimes v-e\id,e).$$
We clearly have $z'\in K_{\gamma,y'}$. Furthermore, using \eqref{eq:squareroot} it follows that
$$|z-z'|\leq\sqrt{2\gamma} \abs{p(y)-p(y')} <  \varepsilon.$$

It is easy to see that this construction is symmetric with respect to $y,y'$, so one can similarly prove that for any $z'\in K_{\gamma,y'}$ there exists $z\in K_{\gamma,y}$ such that $|z-z'|<\varepsilon$.

The boundedness of $\bigcup_{y\in\mathscr{U}}K_{\gamma,y}$ follows from the boundedness of $p$ and the fact that every $K_{\gamma,y}$ is bounded in terms of $\gamma$ and $p(y)$ as observed in Section \ref{sec:infty}.
\end{proof}
\begin{proof}[Proof of Theorem \ref{thm:main_theorem}]
Let $(z,p)$, $z=(v,m,\sigma,e)$ be as stated in Theorem \ref{thm:main_theorem}. We set
$$
\gamma:=\sup_{\mathscr{U}}\gamma_\varepsilon(e)+\norm{e}_{L^\infty(\mathscr{U})}+1<+\infty,
$$ where the function $\gamma_\varepsilon$ is given by Proposition \ref{prop:condition_for_being_in_K_gamma_hull} and $\varepsilon>0$ is taken from \eqref{eq:sufficient_condition_for_Linfty_hull}. By said Proposition there holds $z(x,t)\in \text{int} \left(K^{co}_{\gamma,(x,t)}\right)$ for any $(x,t)\in\mathscr{U}$ and we therefore can conclude by Theorem \ref{thm:abstract_CI_theorem} above.
\end{proof}

\section{The density result}\label{sec:density_proof}

In this section we prove Theorem \ref{thm:dense}, by generalizing some of the techniques from \cite{Sz-MPI}. The proof is split into 3 steps.

\textbf{Step 1.} We begin with the construction of initial data arbitrarily close to any solenoidal $L^2$ vector field, for which there exists a subsolution with turbulence at initial time for any sufficiently small dissipation measure $\mu$.

More precisely, let $\delta>0$
and 
\begin{align}
    B_{\delta,-}=\{\mu\in\cC^0(\T^2\times [0,T]):\ -\delta/T<\mu\leq 0\}.
\end{align}
Contrary to the previous section we will now simultaneously work with different pressures $p\in\cC^0(\T^2\times[0,T])$. In order to emphasize the dependence of the nonlinear constraints on the pressure, we now write $K_{\gamma,p(x,t)}$ instead of $K_{\gamma,(x,t)}$. 

\begin{lemma}\label{lem:in3}
Let $w\in L^2(\T^2)$ with $\divv w=0$. For any $\mu\in B_{\delta,-}$ there exists 
$z=z^\mu\in \cC^0(\T^2\times[0,T])$, $p=p^\mu\in \cC^0(\T^2\times [0,T])$  
solving \eqref{eq:linear_system_with_p}, as well as $\gamma>0$ such that $z(x,t)\in\text{int }(K_{\gamma,p(x,t)}^{co})\text{ for }(x,t)\in\T^2\times[0,T]$, and
\begin{align}\label{eq:nsclose}
\|w-v(\cdot,0)\|_2^2\leq \delta,\quad \int_{\T^2} e(x,0)-\frac{1}{2}|v(x,0)|^2 \, dx\leq \delta.
\end{align}
Furthermore, the family $\left(z^\mu,p^\mu\right)_{\mu\in B_{\delta,-}}$ is uniformly equicontinuous at the time slice $\T^2\times\{0\}$ and $(z^\mu(\cdot,0),p^\mu(\cdot,0))$ is independent of $\mu$, as is $\gamma$.
\end{lemma}
\begin{proof}

Let $\mu\in \mathcal C^0(\T^2\times[0,T])$ with $-\frac{\delta}{T}< \mu\leq 0$, and $w\in L^2(\T^2)$ with $\divv w=0$. Then there exists $w_{\delta}\in \cC^\infty(\T^2)$ with $\divv w_{\delta}=0$ and such that
$$\|w-w_{\delta}\|_2^2\leq \delta.$$
Since we are in the case of two space dimensions, it follows that 
there exists a solution 
 $(\tilde v,\tilde p)\in \cC^0(\T^2\times[0,T])$ to the incompressible Euler equations with initial data $w_\delta$.

Therefore, we may set
\begin{align*}
    v&:=\tilde v,\\
    m&:=\left(\tilde p+\frac{|\tilde v|^2}{2}\right)\tilde v
    ,\\
    \sigma&:=\tilde v\otimes \tilde v-\frac{|\tilde v|^2}{2}\id,\\
    e&:=
    \frac{|\tilde v|^2}{2}+\delta+\int_0^t \mu(x,s)\, ds,\\
    p&:=\tilde p - \delta - \int_0^t \mu(x,s)\, ds,
\end{align*}
to obtain that $z=(v,m,\sigma,e)$ solves \eqref{eq:linear_system_with_p} with pressure $p$ and energy dissipation rate $\mu$. The stated equicontinuity and $\mu$-independence at $t=0$ can easily be observed.

Next one checks that
$m=(e+p)v$ and
$$e-\lamax(v\otimes v-\sigma)=e-\frac{1}{2}|v|^2=\delta+\int_0^t \mu(x,s)>0.$$
 From Proposition \ref{prop:condition_for_being_in_K_gamma_hull} it then follows that there exists $\gamma>0$, which can be chosen independently of $\mu$, such that $z(x,t)\in\text{int}(K_{\gamma,p(x,t)}^{co})\text{ for }(x,t)\in\T^2\times[0,T]$.
 
 Finally, we have
 \begin{align*}
     \int_{\T^2} e(x,0)-\frac{1}{2}|v(x,0)|^2 \, dx=\delta|\T^2|,
 \end{align*}
which finishes the proof of the Lemma.
 \end{proof}

\textbf{Step 2.} In this step we prove a typical perturbation lemma for the initial data and subsolutions constructed in the previous step. For a more detailed proof we refer to
\cite{Crippa,GKSz,Sz-MPI}. Note however that in contrast to the perturbations therein we pay slightly more attention to the fact that we want to use the same perturbation for the whole family of subsolutions from Step 1.
\begin{lemma}\label{lem:in1} Let $\left(z^\mu,p^\mu\right)_{\mu\in B_{\delta,-}}$ and $\gamma>0$ be given by Lemma \ref{lem:in3}.
There exists $C>0$ and a sequence
$\{z_j\}_{j\geq 0}\subset \cC^0(\T^2\times[0,T])$ -- both independent of $\mu$ -- such that
\begin{enumerate}[(i)]
    \item $z^\mu+z_j$ solves \eqref{eq:linear_system_with_p} with $p^\mu$ and $\mu$,
    \item $z^\mu(x,t)+z_j(x,t)\in\text{int }(K_{\gamma,p^\mu(x,t)}^{co})\text{ for }(x,t)\in\T^2\times[0,T]$,
    \item $z_j(\cdot,0)\rightharpoonup 0$
weakly in $L^2(\T^2)$,
\item $\supp(z_j)\subset\T^2\times[0,1/j)$,
\item there holds
$$\int_{\T^2} |z_j(x,0)|^2\, dx \geq C\int_{\T^2} \left(d(z(x,0),K_{\gamma,p(x,0)})\right)^2\, dx .$$
\end{enumerate}
\end{lemma}
\begin{proof}

Let $x_0\in\T^2$. Let us fix some
arbitrary $\mu\in B_{\delta,-}$ and consider the associated $z=z^\mu$, $p=p^\mu$ solving \eqref{eq:linear_system_with_p}.

Since $ z (x_0,0)\in \text{int }(K_{\gamma,p(x_0,0)}^{co})$, it follows
from Corollary \ref{cor:seg} that there exists $\bar z=\bar z (x_0,0) \in\Lambda$ such that
\begin{gather*}
[z(x_0,0)-\bar z,z(x_0,0)+\bar z]\subset \text{int} (K_{\gamma,p(x_0,0)}^{co}),\quad
|\bar z|\geq\frac{1}{2N}d(z(x_0,0),K_{\gamma,p(x_0,0)}).
\end{gather*}
Consequently, there exists $\rho=\rho(x_0,0) >0$ such that
$$[z(x_0,0)-\bar z,z(x_0,0)+\bar z]+\overline B_{2\rho }(0) \subset \text{int} (K_{\gamma,p(x_0,0)}^{co}).$$
Let us observe that since $z(x_0,0)$ and $p(x_0,0)$ are independent of $\mu$ by virtue of Lemma \ref{lem:in3}, so are the quantities $\bar{z}$ and $\rho$.

Next using the equicontinuity of the family $(p^\mu)_{\mu\in B_{\delta,-}}$ at $(x_0,0)$ one can check as in Lemma \ref{lem:cont_of_constr} that the family of maps $\left((x,t)\mapsto K_{\gamma,p^\mu(x,t)}\right)_{\mu\in B_{\delta,-}}$ is equicontinuous at $(x_0,0)$ with respect to the Hausdorff metric $d_{\cH}$.
From this property and the equicontinuity of $(z^\mu)_\mu$ at $(x_0,0)$, it then follows that there exists $r_0=r_0(x_0,0)>0$ independent of $\mu$ such that for any $(x,t)\in B_{r_0}((x_0,0))$, there holds
$$[z(x,t)-\bar z,z(x,t)+\bar z]+\overline B_{\rho }(0) \subset \text{int} (K_{\gamma,p(x,t)}^{co}).$$

Using Lemma \ref{lem:locpw}, in particular Remark \ref{rem:plane_waves_at_time_0}, we deduce that there exists a sequence $z_k\in \cC_c^\infty(B_{r_0}(x_0)\times[0,1/k))$ which solves \eqref{eq:linear_system} and satisfies
\begin{itemize}
\item[(a)] $z(x,t)+z_k(x,t)\in \text{int }(K_{\gamma,p(x,t)}^{co})$ for every $(x,t)\in\T^2\times[0,T]$,
\item[(b)] $z_k(\cdot,0)\rightharpoonup 0$ in $L^2(\T^2)$,
\item[(c)] $\int_{\T^2} |z_k(x,0)|^2\, dx \geq C|B_{r_0}(x_0)|\left(d(z(x_0,0),K_{\gamma,p(x_0,0)})\right)^2$
\end{itemize}
with a constant $C>0$ independent of $x_0$, $r_0$ and $z$. A corresponding sequence can also be found for any $B_r(x_0)$ with $r<r_0$. Furthermore, $z_k$ are independent of $\mu$ by the above considerations.


Now we proceed by a standard covering argument. We may find finitely many disjoint balls $B_i:=B_{r_i}(x_i)$, $i=1,\ldots, I$, and associated sequences $z_k^i$ solving \eqref{eq:linear_system} and satisfying (a)-(c) for each $i,k$, as well as
$$\supp(z_k^i)\subset\T^2\times[0,1/k),$$
and
\begin{align}\label{eq:balld}\int_{\T^2} \left(d(z(x,0),K_{\gamma,p(x,0)})\right)^2\, dx \leq 2\sum_{i=1}^I |B_i| \left(d(z(x_i,0),K_{\gamma,p(x_i,0)})\right)^2.
\end{align}
We define $$z_j:=\sum_{i=1}^I z^i_j,
$$
then it is easy to see that (i)-(iv) are satisfied.
Finally, from (c) and \eqref{eq:balld} it follows that
$$\int_{\T^2} |z_j(x,0)|^2\, dx \geq \frac{C}{2}\int_{\T^2} \left(d(z(x,0),K_{\gamma,p(x,0)})\right)^2\, dx. $$
\end{proof}

\begin{remark}\label{rem:in4}
Note that for any $j\geq 0$ we obtain a new family $(z^\mu+z_j,p^\mu)_{\mu\in B_{\delta,-}}$ having the properties stated in Lemma \ref{lem:in3} except for \eqref{eq:nsclose}, which has not been used in the proof of Lemma \ref{lem:in1}. In consequence Lemma \ref{lem:in1} can also be applied to all these new families.
\end{remark}

\textbf{Step 3.} The next Lemma essentially shows that with any initial data belonging to a subsolution which is turbulent, i.e. not an Euler state, at initial time, one may associate an initial data belonging to a subsolution which is an Euler state at initial time, and the square of the $L^2$-norm of the difference of the corresponding initial velocities is controlled by the deviation to an Euler state of the first initial data.


\begin{lemma}\label{lem:in2}
Let $(z(\cdot,0),p(\cdot,0))$, $\gamma>0$ be as given by Lemma \ref{lem:in3} (independent of any chosen dissipation measure $\mu\in B_{\delta,-}$), and set
$(w,f):=(v,e)(\cdot,0)$. Then there exists  $(\bar w,\bar f)\in L^2(\T^2;\R^3)$ such that for any $\mu\in B_{\delta,-}$ there exists a function $\bar z^\mu\in \cC^0(\T^2\times(0,T])$ solving \eqref{eq:linear_system_with_p} with pressure $p^\mu$, dissipation $\mu$ and initial data $(\bar v^\mu,\bar e^\mu)(\cdot,0)=(\bar w,\bar f)$, such that $\bar z^\mu(x,t)\in\text{int }( K_{\gamma,p^\mu(x,t)}^{co})$ for $(x,t)\in\T^2\times(0,T]$, $\divv \bar w =0,$
$\bar f=\frac{1}{2}|\bar w|^2$ a.e. in $\T^2$, and
$$\|w-\bar{w}\|_{L^2(\T^2)}^2\leq 9\int_{\T^2} f(x)-\frac{1}{2}|w(x)|^2 \, dx.
$$
\end{lemma}
\begin{proof}
Let $\nu:=\int_{\T^2} f(x)-\frac{1}{2}|w(x)|^2$. Note that $\nu>0$, because otherwise there holds $z(x,0)\in \partial K^{co}_{\gamma,p(x,0)}$ in contradiction with the fact that the subsolutions $z^\mu$ given by Lemma \ref{lem:in3} are turbulent at initial time.

We construct the following sequences recursively. Set $(w_0,f_0):=(w,f)$ and let $(z_0^\mu,p_0^\mu)_{\mu\in B_{\delta,-}}$ be the family of subsolutions from Lemma \ref{lem:in3}. We recall that $z_0(\cdot,0)$ is independent of $\mu$, as in the previous lemma.

Then, given the corresponding family for $k\geq 0$, one may use Lemma \ref{lem:in1}, cf. Remark \ref{rem:in4}, to deduce the existence of a sequence of subsolutions $(z_k^\mu+z_{k,j})_{j\geq 0}$ such that $z_{k,j}(\cdot,0)\rightharpoonup 0$ in $L^2(\T^2)$ as $j\to+\infty$. Using the identity
\begin{align*}\|z_{k,j}(\cdot,0)\|_2^2+&\|z_{k}(\cdot,0)\|_2^2 -\|z_k(\cdot,0)+z_{k,j}(\cdot,0)\|_2^2 =-2\langle z_{k,j}(\cdot,0) ,z_{k}(\cdot,0) \rangle \to 0,
\end{align*}
one may define $z^\mu_{k+1}:=z_k^\mu+z_{k,j}$, $(w_{k+1},f_{k+1}):=(w_k+v_{k,j}(\cdot,0),f_k+e_{k,j}(\cdot,0))$ for a $j>0$ large enough such that one has
\begin{gather}
\|z_{k+1}(\cdot,0)-z_{k}(\cdot,0)\|_2^2\leq \|z_{k+1}(\cdot,0)\|_2^2-\|z_{k}(\cdot,0)\|_2^2+2^{-k}\label{eq:itseq1},\\
\|w_{k+1}-w_{k}\|_2^2\leq\|w_{k+1}\|_2^2 -\|w_{k}\|_2^2+\nu 2^{-k},\label{eq:itseq2}\\ \left|\int_{\T^2} f_{k+1}(x)-f_k(x) \, dx\right|\leq  \nu 2^{-k}\label{eq:itseq3},
\end{gather}
and in addition, due to Lemma \ref{lem:in1}, one also has
\begin{align}
    \int_{\T^2} |z_{k+1}(\cdot,0)(x)-z_{k}(\cdot,0)|^2\, dx &\geq C\int_{\T^2} \left(d(z_k(x,0),K_{\gamma,p(x,0)})\right)^2\, dx,\label{eq:itdis}\\ \supp(z^\mu_{k+1}-z^\mu_k)&\subset\T^2\times[0,1/k).\label{eq:itsup}
\end{align}
From \eqref{eq:itseq1} it follows that for any $l>k$ one has
\begin{align*}
   \|z_{l}(\cdot,0)-z_{k}(\cdot,0)\|_2^2\leq 2 \left(\|z_{l}(\cdot,0)\|_2^2 -\|z_{k}(\cdot,0)\|_2^2+ 2^{-k+1}\right).
\end{align*}
By the weak convergence of $(z_{k,j})_j$ we may in addition assume that $\|z_{k}(\cdot,0)\|_2$ is increasing (up to a sequence converging to $0$), and since it is bounded due to $z_k(\cdot,0)\in K_{\gamma,p(\cdot,0)}^{co}$, it has to converge. Consequently, $z_k(\cdot,0)$ is a Cauchy sequence in $L^2(\T^2)$ and it converges strongly to some $\bar \zeta=(\bar w,\bar u, \bar S,\bar f)\in L^2(\T^2;Z)$, which is once more independent of $\mu$, since  $z_k(\cdot,0)$ were also independent of $\mu$.

Using \eqref{eq:itsup}, it follows that for any $\mu\in B_{\delta,-}$ there exists $\bar z^\mu\in \cC^0(\T^2\times (0,T];Z)$ such that
$$z^\mu_k\to \bar z^\mu\text{ in }\cC^0_{loc}(\T^2\times (0,T];Z).$$
Hence, $\bar z^\mu$ solves \eqref{eq:linear_system_with_p} w.r.t. initial data $(\bar v^\mu(\cdot,0),\bar e^\mu(\cdot,0))=(\bar w,\bar f)$, pressure $p^\mu$ and dissipation measure $\mu$. Moreover, $\bar z^\mu(x,t)\in\text{int }( K_{\gamma,p^\mu(x,t)}^{co})$ for $(x,t)\in\T^2\times(0,T]$. From \eqref{eq:itdis} it follows that $d(\bar z (\cdot,0),K_{\gamma,p(\cdot,0)})=0$ a.e., therefore $\frac{1}{2}|\bar w|^2=\bar f$ a.e. in $\T^2$.

Finally, from \eqref{eq:itseq2} and  \eqref{eq:itseq3}  one obtains that
\[
     \|\bar{w}-w\|_2^2\leq 2 \left(\|\bar w\|_2^2 -\|w\|_2^2\right)+\nu,\quad \left|\int_{\T^2}\bar f(x)-f(x) \, dx\right|\leq \nu,
\]     
which implies
\begin{align*}
     \|\bar{w}-w\|_2^2&\leq 4 \int_{\T^2}\bar  f(x)-\frac{1}{2}|w(x)|^2 \, dx+\nu \leq 4 \int_{\T^2} f(x)-\frac{1}{2}|w(x)|^2 \, dx +5\nu \\&= 9 \int_{\T^2} f(x)-\frac{1}{2}|w(x)|^2 \, dx.
\end{align*}
Since $(\bar w,\bar f)$ is independent of $\mu$ by construction, this finishes the proof of Lemma \ref{lem:in2}.
\end{proof}

\textbf{Conclusion.} Observe that, by virtue of Theorem \ref{thm:abstract_CI_theorem}, the function $\bar w$
given by Lemma \ref{lem:in2} above is an initial data having the property that for any dissipation measure $\mu\in B_{\delta,-}$ there exists a corresponding subsolution. Note also that all associated subsolutions have $\mathscr{U}=\T^2\times (0,T]$ as their turbulent zone.

\begin{proof}[Proof of Theorem \ref{thm:dense}] Now we are in a position to finish the proof of Theorem \ref{thm:dense}.
Let $\delta>0$ and $w\in L^2(\T^2)$ with $\divv w=0$, we apply Lemma \ref{lem:in3} to deduce for any $\mu\in B_{\delta,-}$ the existence of $p^\mu\in \cC^0(\T^2\times [0,T])$ and a subsolution $z^\mu$ which satisfy \eqref{eq:nsclose} and $z^\mu(x,t)\in\text{int }(K_{\gamma,p^\mu(x,t)}^{co})\text{ for }(x,t)\in\T^2\times[0,T]$ for some $\gamma>0$. 

We then have $(z^\mu(\cdot,0),p^\mu(\cdot,0))$ independent of $\mu$, and by Lemma \ref{lem:in2} there exists wild initial data $\bar w\in L^2(\T^2)$ and for each $\mu\in B_{\delta,-}$ an associated subsolution $\bar{z}^\mu$ with turbulent zone $\T^2\times(0,T]$. Furthermore there holds
$$\|v(\cdot,0)-\bar{w}\|_2^2\leq 9\int_{\T^2} e(x,0)-\frac{1}{2}|v(x,0)|^2 \, dx.$$
Hence, we may use \eqref{eq:nsclose} to conclude that
\begin{align*}
    \|w-\bar{w}\|_2^2&\leq 2\|v(\cdot,0)-\bar{w}\|_2^2+2\|v(\cdot,0)-w\|_2^2\\ &\leq 18\int_{\T^2} e(x,0)-\frac{1}{2}|v(x,0)|^2 \, dx + 2\delta \leq 20 \delta.
\end{align*}
Since $\delta>0$ was arbitrary, this concludes the proof of Theorem \ref{thm:dense}.
\end{proof}





%

%

\vspace{30pt}
\noindent Mathematisches Institut,  Universit\"at Leipzig,  Augustusplatz 10, D-04109 Leipzig \\
\texttt{bjoern.gebhard@math.uni-leipzig.de}\\
\texttt{jozsef.kolumban@math.uni-leipzig.de}
%
\end{document}